\documentclass[12pt]{amsart}
\usepackage{amsmath,amsthm,amscd,amsfonts, amssymb}
\usepackage{color}
\newcommand{\sect}[1]{\section{#1}\setcounter{equation}{0}}
\newcommand{\subsect}[1]{\subsection{#1}}

\font\mbn=msbm10 scaled \magstep1
\font\mbs=msbm7 scaled \magstep1
\font\mbss=msbm5 scaled \magstep1
\newfam\mbff
\textfont\mbff=\mbn
\scriptfont\mbff=\mbs
\scriptscriptfont\mbff=\mbss\def\mbf{\fam\mbff}

\def\Z{{\mbf Z}}
\def\Co{{\mbf C}}

\def\Di{{\mbf D}}
\def\N{{\mbf N}}

\newtheorem{Th}{Theorem}[section]
\newtheorem{Lm}[Th]{Lemma}
\newtheorem{C}[Th]{Corollary}

\newtheorem{Proposition}[Th]{Proposition}
\newtheorem{R}[Th]{Remark}

\newtheorem*{Problem}{Problem}
\newtheorem*{Problem 1}{Problem 1}
\begin{document}

\title[Holomorphic Banach Vector Bundles]{Holomorphic Banach Vector Bundles on the Maximal Ideal Space of H$^\infty$ and the Operator Corona Problem of Sz.-Nagy}

\author{Alexander Brudnyi} 
\address{Department of Mathematics and Statistics\newline
\hspace*{1em} University of Calgary\newline
\hspace*{1em} Calgary, Alberta\newline
\hspace*{1em} T2N 1N4}
\email{albru@math.ucalgary.ca}
\keywords{Holomorphic Banach vector bundle, maximal ideal space, operator corona problem, $\bar\partial$-equation}
\subjclass[2000]{Primary 30D55. Secondary 30H05}

\thanks{Research supported in part by NSERC}

\begin{abstract}
We establish triviality of some holomorphic Banach vector bundles on the maximal ideal space $M(H^\infty)$ of the Banach algebra $H^\infty$ of bounded holomorphic functions on the unit disk $\Di\subset\Co$ with pointwise multiplication and supremum norm. We apply the result to the study of the Sz.-Nagy operator corona problem.
\end{abstract}

\date{}



\maketitle

\sect{Formulation of Main Results}
\subsect{}
We continue our study started in \cite{Br2} of analytic objects on the maximal ideal space $M(H^\infty)$ of the Banach algebra $H^\infty$ of bounded holomorphic functions on the unit disk $\Di\subset\Co$ with pointwise multiplication and supremum norm.
The present paper deals with holomorphic Banach vector bundles defined on $M(H^\infty)$ and the operator corona problem posed by Sz.-Nagy.  Recall that for a commutative unital complex Banach algebra $A$ with dual space $A^*$ the {\em maximal ideal space} $M(A)$ of $A$ is the set of nonzero homomorphisms $A\to\Co$ equipped with the {\em Gelfand topology}, the weak$^*$ topology induced by $A^*$. It is a compact Hausdorff space contained in the unit ball of $A^*$. In the case of $H^\infty$ evaluation at a point of $\Di$ is an element of $M(H^\infty)$, so $\Di$ is naturally embedded into $M(H^\infty)$ as an open subset. The famous Carleson corona theorem \cite{C} asserts that $\Di$ is dense in $M(H^\infty)$.

Let $U\subset M(H^\infty)$ be an open subset and $X$ be a complex Banach space. 
A continuous function $f\in C(U;X)$ is said to be $X$-valued {\em holomorphic} if its restriction to $U\cap\Di$ is $X$-valued holomorphic in the usual sense. 

By $\mathcal O(U;X)$ we denote the vector space of $X$-valued holomorphic functions on $U$. 

Let $E$ be a continuous Banach vector bundle on $M(H^\infty)$ with fibre $X$ defined on an open cover $\mathcal U=(U_i)_{i\in I}$ of $M(H^\infty)$ by a cocycle $\{g_{ij}\in C(U_i\cap U_j; GL(X))\}$; here $GL(X)$ is the group of invertible elements of the Banach algebra $L(X)$ of bounded linear operators on $X$ equipped with the operator norm. We say that $E$ is holomorphic if all $g_{ij}\in  \mathcal O(U_i\cap U_j; GL(X))$.  In this case $E|_\Di$ is a holomorphic Banach vector bundle on $\Di$ in the usual sense.
Recall that $E$ is defined as the quotient space of the disjoint union $\sqcup_{i\in I}U_i\times X$ by the equivalence relation:
\[
U_j\times X\ni u\times x\sim u\times g_{ij}(u)x\in U_i\times X.
\]
The projection $p:E\to X$ is induced by natural projections $U_i\times X\to U_i$, $i\in I$.

A morphism $\varphi : (E_1, X_1, p_1)\to (E_2, X_2, p_2)$ of holomorphic Banach vector bundles on $M(H^\infty)$ is a continuous map which sends each vector space $p_1^{-1}(w)\cong X_1$ linearly to vector space $p_2^{-1}(w)\cong X_2$, $w\in M(H^\infty)$, and such that $\varphi|_\Di: E_1|_\Di\to E_2|_\Di$ is a holomorphic map of complex Banach manifolds. If, in addition, $\varphi$ is bijective, then $\varphi$ is called an isomorphism. 

We say that a holomorphic Banach vector bundle $(E,X,p)$ on $M(H^\infty)$ is {\em holomorphically trivial} if it is isomorphic to the trivial bundle
$M(H^\infty)\times X$. (For the basic facts of the theory of bundles, see, e.g., \cite{Hus}.)

Let $GL_0(X)$ be the connected component of $GL(X)$ containing the identity map $I_X:={\rm id}_X:X\to X$. Then $GL_0(X)$ is a clopen normal subgroup of $GL(X)$. By $q:GL(X)\to GL(X)/GL_0(X):=C(GL(X))$ we denote the continuous quotient homomorphism onto the discrete group of connected components of $GL(X)$. Let $E\to M(H^\infty)$ be a holomorphic Banach vector bundle with fibre $X$ defined on a finite open cover $\mathcal U=(U_i)_{i\in I}$ of $M(H^\infty)$ by a cocycle $g=\{g_{ij}\in \mathcal O (U_i\cap U_j; GL(X))\}$. By $E_{C(GL(X))}$ we denote the principal bundle on $M(H^\infty)$ with fibre $C(GL(X))$ defined on $\mathcal U$ by the locally constant cocycle $q(g)=\{q(g_{ij})\in C(U_i\cap U_j; C(GL(X)))\}$.

\begin{Th}\label{bundle}
$E$ is holomorphically trivial if and only if the associated bundle $E_{C(GL(X))}$ is trivial in the category of principal bundles with discrete fibres. 
\end{Th}
\begin{C}\label{cor1.2}
$E$ is holomorphically trivial in one of the following cases:
\begin{itemize}
\item[(1)]
The image of each function $g_{ij}$ in the definition of $E$ belongs to $GL_0(X)$ (e.g., this is true if $GL(X)$ is connected);
\item[(2)]
$E$ is trivial in the category of continuous Banach vector bundles.
\end{itemize}
\end{C}
In particular, the result is valid for spaces $X$ with contractible group $GL(X)$.
The class of such spaces include infinite-dimensional Hilbert spaces, spaces $\ell^p$ and $L^p[0,1]$, $1\le p\le \infty$, $c_0$ and $C[0,1]$, spaces $L_p(\Omega,\mu)$, $1<p<\infty$, of $p$-integrable measurable functions on an arbitrary
measure space $\Omega$, and some classes of reflexive symmetric function spaces; the class of spaces $X$ with connected but not simply connected group $GL(X)$ include finite dimensional Banach spaces, finite direct products of James spaces etc., see, e.g., \cite{M} and references therein. There are also Banach spaces $X$ whose linear groups $GL(X)$ are not connected. E.g., the groups of connected components of spaces $\ell^{p}\times\ell^q$, $1\le p<q<\infty$, are isomorphic to $\Z$, see  \cite{Do}.

We deduce Theorem \ref{bundle} from more general results on triviality of holomorphic principal bundles on $M(H^\infty)$, see Theorems \ref{principal1}, \ref{bundle1}. 
\subsect{}
We apply Theorem \ref{bundle} to the Sz.-Nagy operator corona problem \cite{SN} posed in 1978. In its formulation $H^\infty(L(X,Y))$ stands for the Banach space of holomorphic functions $F$ on $\Di$ with values in the space of bounded linear operators $X\to Y$ of complex Banach spaces $X,Y$ with norm
$\|F\|:=\sup_{z\in\Di}\|F(z)\|_{L(X,Y)}$. 
\begin{Problem}[Sz.-Nagy]
Let $F\in H^\infty(L(H_1,H_2))$, where $H_i$, $i=1,2$, are separable Hilbert spaces,  satisfy $\|F(z)x\|\ge\delta\|x\|$ for every $x\in H_1$ and every $z\in\Di$, where $\delta>0$ is a constant. Does there exist $G\in H^\infty (L(H_2,H_1))$ such that $G(z)F(z)=I_{H_1}$ for every $z\in\Di$?
\end{Problem}

This problem is of great interest in operator theory (angles between invariant
subspaces, unconditionally convergent spectral decompositions), as well as in control theory. It is also related
to the study of submodules of $H^\infty$ and to many other subjects of analysis. Obviously, the condition imposed on $F$ is
necessary. It implies existence of a uniformly bounded family of left inverses of
$F(z)$, $z\in\Di$. The question is whether this condition is sufficient for the existence of
a bounded analytic left inverse of $F$. In general, the answer is known to be negative (see
\cite{T1}, \cite{T2}, \cite{TW} and references therein). But in some specific cases it is positive. In particular, Carleson's theorem stating that a Bezout equation $\sum_{i=1}^n g_i f_i=1$ is solvable with $\{g_i\}_{i=1}^n\subset H^\infty$ as soon as $\{f_i\}_{i=1}^n\subset H^\infty$ satisfies $\max_{1\le i\le n}|f_i(z)|>\delta>0$ for every $z\in\Di$ means that the answer is positive when ${\rm dim}\, H_1=1$, ${\rm dim}\, H_2=n<\infty$.
More generally, the answer is positive as soon as ${\rm dim}\, H_1<\infty$. For a long time there were no 
positive results in the case ${\rm dim}\, H_1 =\infty$. The first positive results in this
case were obtained in \cite{V}. Following this paper consider a more general 
\begin{Problem 1}
Let $X_1,X_2$ be complex Banach spaces and  $F\in H^\infty(L(X_1,X_2))$ be such that for each $z\in\Di$ there exists a left inverse $G_z$ of $F(z)$ satisfying $\sup_{z\in\Di}\|G_z\|<\infty$. Does there exist $G\in H^\infty(L(X_2,X_1))$ such that $G(z)F(z)=I_{X_1}$ for every $z\in\Di$?
\end{Problem 1}
Since in this general setting the answer is negative, as in \cite{V} we restrict ourselves to the case of $F\in H_{\rm comp}^\infty(L(X_1,X_2))$, the space of holomorphic functions on $\Di$ with relatively compact images in $L(X_1,X_2)$. Then the answer is positive for 
$F$ that can be uniformly approximated by finite sums $\sum f_k(z)L_k$,
where $f_k\in H^\infty$ and $L_k\in L(X_1,X_2)$, see \cite[Th. 2.1]{V}. The question of whether each $F\in H_{\rm comp}^\infty(L(X_1,X_2))$ can be obtained in that form is closely related to the still open problem about the Grothendieck approximation property for $H^\infty$.

Now, under the above restriction we obtain a much stronger result.
\begin{Th}\label{complem}
Let $F\in H_{\rm comp}^\infty(L(X_1,X_2))$, where $X_i$, $i=1,2$, are complex Banach spaces, be such that for every $z\in\Di$ there exists a left inverse $G_z$ of $F(z)$ satisfying $\sup_{z\in\Di}\|G_z\|<\infty$. Let $Y:={\rm Ker}\, G_0$. Assume that $GL(Y)$ is connected. Then
there exist functions $H\in H_{\rm comp}^\infty(L(X_1\oplus Y, X_2))$, $G\in H_{\rm comp}^\infty(L(X_2, X_1\oplus Y))$ such that $H(z)G(z)=I_{X_2}$, $G(z)H(z)=I_{X_1\oplus Y}$ and $H(z)|_{X_1}=F(z)$ for all $z\in\Di$.
\end{Th}
\begin{C}\label{cor1}
Let $F\in H_{\rm comp}^\infty(L(H_1,H_2))$, where $H_i$, $i=1,2$, are Hilbert spaces,  satisfy $\|F(z)x\|\ge\delta\|x\|$ for every $x\in H_1$ and every $z\in\Di$, where $\delta>0$ is a constant. Let $Y:=\bigl(F(0)(H_1)\bigr)^{\bot}$. Then there exist functions $H\in H_{\rm comp}^\infty(L(H_1\oplus Y, H_2))$, $G\in H_{\rm comp}^\infty(L(H_2, H_1\oplus Y))$ such that $H(z)G(z)=I_{H_2}$, $G(z)H(z)=I_{H_1\oplus Y}$ and $H(z)|_{H_1}=F(z)$ for all $z\in\Di$.
\end{C}
This follows immediately from Theorem \ref{complem} because in the Hilbert case the condition of the corollary implies existence of a uniformly bounded family of left inverses of $F(z)$, $z\in\Di$.

Finally, we obtain the positive answer in Problem 1 for spaces $H_{\rm comp}^\infty$.
\begin{Th}\label{teo1.5}
Let $X_1,X_2$ be complex Banach spaces and  $F\in H_{\rm comp}^\infty(L(X_1,X_2))$ be such that for every $z\in\Di$ there exists a left inverse $G_z$ of $F(z)$ satisfying $\sup_{z\in\Di}\|G_z\|<\infty$. Then there exist $G\in H_{\rm comp}^\infty(L(X_2,X_1))$ such that $G(z)F(z)=I_{X_1}$ for every $z\in\Di$.
\end{Th}

\medskip

{\em Acknowledgment.} I thank N. K. Nikolski for useful discussions.
 
\sect{Preliminary Results}
In this part we collect some preliminary results used in the proof of Theorem \ref{bundle}.
\subsection{Runge-type approximation theorem}

A compact subset $K\subset M(H^\infty)$ is called {\em holomorphically convex} if for any $x\notin K$ there is $f\in H^\infty$ such that
\[\max_{K}|f|<|f(x)|.\]
\begin{Proposition}[\cite{Br2}, Lemma 5.1]\label{holconv}
Let $N\subset M(H^\infty)$ be an open neigbourhood of a holomorphically convex set $K$. Then there exists an open set $U\Subset N$ 
containing $K$ such that the closure $\bar U$ is holomorphically convex.
\end{Proposition}
\begin{Th}[\cite{Br2}, Theorem 1.7]\label{runge}
Let $B$ be a complex Banach space.
Any $B$-valued holomorphic function defined on a neigbourhood of a holomorphically compact set $K\subset M(H^\infty)$ can be uniformly approximated on $K$ by functions from $\mathcal O(M(H^\infty);B)$. 
\end{Th}

\subsection{Maximal Ideal Space of $H^\infty$}

\subsubsection{}
Recall that the pseudohyperbolic metric on $\Di$ is defined by
\[
\rho(z,w):=\left|\frac{z-w}{1-\bar w z}\right|,\qquad z,w\in\Di.
\]
For $x,y\in\mathcal M(H^\infty)$ the formula
\[
\rho(x,y):=\sup\{|\hat f(y)|\, ;\, f\in H^\infty,\, \hat f(x)=0,\, \|f\|\le 1\}
\]
gives an extension of $\rho$ to $M(H^\infty)$.
The {\em Gleason part} of $x\in\mathcal M$ is then defined by $\pi(x):=\{y\in M(H^\infty)\, ;\, \rho(x,y)<1\}$. For $x,y\in M(H^\infty)$ we have
$\pi(x)=\pi(y)$ or $\pi(x)\cap\pi(y)=\emptyset$. Hoffman's classification of Gleason parts \cite{H} shows that there are only two cases: either $\pi(x)=\{x\}$ or $\pi(x)$ is an analytic disk. The former case means that there is a continuous one-to-one and onto map $L_x:\Di\to\pi(x)$ such that
$\hat f\circ L_x\in H^\infty$ for every $f\in H^\infty$. Moreover, any analytic disk is contained in a Gleason part and any maximal (i.e., not contained in any other) analytic disk is a Gleason part. By $M_a$ and $M_s$ we denote the sets of all non-trivial (analytic disks) and trivial (one-pointed) Gleason parts, respectively. It is known that $M_a\subset M(H^\infty)$ is open. Hoffman proved that $\pi(x)\subset M_a$ if and only if $x$ belongs to the closure of some interpolating sequence in $\Di$. 

\subsubsection{Structure of $M_a$}
In \cite{Br1} $M_a$ is described as a fibre bundle over a compact Riemann surface. Specifically, let $G$ be the fundamental group of a compact Riemann surface $S$ of genus $\ge 2$. Let $\ell_\infty(G)$ be the Banach algebra of bounded complex-valued functions on $G$ with pointwise multiplication and supremum norm. By $\beta G$ we denote the {\em Stone-\v{C}ech compactification} of $G$, i.e., the maximal ideal space of $\ell_\infty(G)$ equipped with the Gelfand topology.
 
The universal covering $r:\Di\to S$ is a principal fibre bundle with fibre $G$. Namely, there exists a finite open cover ${\mathcal U} = (U_i)_{i\in I}$ of $S$ by sets biholomorphic to $\Di$ and a locally constant cocycle $\bar g=\{g_{ij}\}\in Z^1({\mathcal U}; G)$ such that $\Di$ is biholomorphic to the quotient space of the disjoint union $V=\sqcup_{i\in I}U_i\times G$ by the equivalence relation $U_i\times G\ni (x, g)\sim (x, gg_{ij})\in U_j\times G$. The identification space is a fibre bundle with projection $r:\Di\to S$ induced by projections $U_i\times G\to U_i$, see, e.g., \cite[Ch.~1]{Hi}.

Next, the right action of $G$ on itself by multiplications is extended to the right continuous action of $G$ on $\beta G$. Let $\tilde r: E(S,\beta G)\to S$ be the associated with this action bundle on $S$ with fibre $\beta G$ constructed by cocycle $\bar g$. Then $E(S,\beta G)$ is a {\em compact} Hausdorff space homeomorphic to the quotient space of the disjoint union $\widetilde V=\sqcup_{i\in I}U_i\times \beta G$ by the equivalence relation $U_i\times \beta G\ni (x, \xi)\sim (x, \xi g_{ij})\in U_j\times \beta G$. The projection $\tilde r:E(S,\beta G)\to S$ is induced by projections $U_i\times \beta G\to U_i$.
Note that there is a natural embedding $V\hookrightarrow\widetilde V$ induced by the embedding $G\hookrightarrow\beta G$. This embedding commutes
with the corresponding equivalence relations and so determines an embedding of $\Di$ into $E(S,\beta G)$ as an open dense subset.
Similarly, for each $\xi\in\beta G$ there exists a continuous injection  $V\to\widetilde V$ induced by the injection $G\to\beta G$, $g\mapsto\xi g$,
commuting with the corresponding equivalence relations. Thus it determines a continuous injective map $i_{\xi}:\Di\to E(S,\beta G)$.
Let $X_G:=\beta G/G$ be the set of co-sets with respect to the right action of $G$ on $\beta G$. Then $i_{\xi_1}(\Di)=i_{\xi_2}(\Di)$ if and only if $\xi_1$ and $\xi_2$ determine the same element of $X_G$.  If $\xi$ represents an element $x\in X_G$, then we write $i_x(\Di)$ instead of $i_{\xi}(\Di)$. In particular, $E(S,\beta G)=\sqcup_{x\in X_G}i_x(\Di)$.

Let $U\subset E(S,\beta G)$ be open. We say that a function $f\in C(U)$ is holomorphic if $f|_{U\cap\Di}$ is holomorphic in the usual sense. The set of holomorphic on $U$ functions is denoted by $\mathcal O(U)$. It was shown in \cite[Th.~2.1]{Br1} that each $h\in H^\infty(U\cap\Di)$ is extended to a unique holomorphic function $\hat h\in \mathcal O(U)$. In particular, the restriction map $\mathcal O(E(S,\beta G))\to H^\infty(\Di)$ is an isometry of Banach algebras. Thus the quotient space of $E(S,\beta G)$ (equipped with the factor topology) by the equivalence relation
$x\sim y \Leftrightarrow f(x)=f(y)$ for all $f\in \mathcal O(E(S,\beta G))$ is homeomorphic to $M(H^\infty)$. By $q$ we denote the quotient map 
$E(S,\beta G)\to M(H^\infty)$.

A sequence $\{g_n\}\subset G$ is said to be interpolating if $\{g_n(0)\}\subset\Di$ is interpolating for $H^\infty$ (here $G$ acts on $\Di$ by M\"{o}bius transformations). Let $G_{in}\subset\beta G$ be the union of closures of all interpolating sequences in $G$. It was shown that $G_{in}$ is an open dense subset of $\beta G$ invariant with respect to the right action of $G$.
The associated with this action bundle $E(S, G_{in})$ on $S$ with fibre $G_{in}$ constructed by the cocycle $\bar g\in Z^1(\mathcal U ;G)$ is an open dense subbundle of $E(S,\beta G)$ containing $\Di$. It was established in \cite{Br1} that 
$q$ maps $E(S, G_{in})$ homeomorphically onto $M_a$ so that for each $\xi\in G_{in}$ the set $q\bigl(i_{\xi}(\Di)\bigr)$ coincides with the Gleason part $\pi\bigl(q(i_\xi (0))\bigr)$.  Also, for distinct $x,y\in E(S,\beta G)$ with $x\in E(S, G_{in})$ there exists $f\in  \mathcal O(E(S,\beta G))$ such that $f(x)\ne f(y)$. Thus $q(x)=x$ for all $x\in E(S,G_{in})$, i.e.,  $E(S,G_{in})=M_a$.

From the definition of $E(S, \beta G)$ follows that for a simply connected open subset $U\subset S$ restriction $E(S, \beta G)|_U\, \bigl(:=\tilde r^{-1}(U)\bigr)$ is a trivial bundle, i.e., there exists an isomorphism of bundles (with fibre $\beta G$) 
$\varphi: E(S, \beta G)|_U\to U\times \beta G$, $\varphi(x):=(\tilde r(x),\tilde\varphi(x))$, $x\in E(S,\beta G)|_U$, mapping $\tilde r^{-1}(U)\cap\Di$ biholomorphically onto $U\times G$.
A subset $W\subset\tilde r^{-1}(U)$ of the form $R_{U,H}:=\varphi^{-1}(U\times H)$, $H\subset \beta G$, is called {\em rectangular}. The base of topology on $E(S, G_{in})\, (:=M_a)$ consists of rectangular sets $R_{U,H}$ with $U\subset S$ biholomorphic to $\Di$ and $H\subset G_{in}$ being the closure of an interpolating sequence in $G$ (so $H$ is a clopen subset of $\beta G$). Another base of topology on $M_a$ is given by sets of the form $\{x\in M_a\, ;\, |\hat B(x)|<\varepsilon\}$, where $B$ is an interpolating Blaschke product (here $\hat B$ is the extension of $B$ to $M(H^\infty)$ by means of the Gelfand transform). This follows from the fact that for a sufficiently small $\varepsilon$ the set $B^{-1}(\Di_\varepsilon)\subset\Di$, $\Di_\varepsilon:=\{z\, ;\, |z|<\varepsilon\}$, is biholomorphic to $\Di_\varepsilon\times B^{-1}(0)$, see \cite[Ch.~X, Lm.~1.4]{Ga}. Hence,
$\{x\in M_a\, ;\, |\hat B(x)|<\varepsilon\}$ is biholomorphic to $\Di_\varepsilon\times \hat B^{-1}(0)$.
\subsubsection{Structure of $M_s$} It was proved in \cite{S2}, that the set $M_s$ of trivial Gleason parts is totally disconnected, i.e., ${\rm dim}\, M_s=0$ (because $M_s$ is compact). 

Also, it was proved in \cite{S1} that the covering dimension of $M(H^\infty)$ is $2$ and the \v{C}ech cohomology group $H^2(M(H^\infty),\Z)=0$.

\subsection{$\bar\partial$-equations with Support in $M_a$}
Let $X$ be a complex Banach space and $U\subset S$ be open simply connected. Let $\varphi: E(S,\beta G)|_U\to U\times \beta G$ be a trivialization as in subsection 2.2.2. We say that a function $f\in C(E(S,\beta G)|_U;X)$ belongs to the space $C^k(E(S,\beta G)|_U;X)$, $k\in\N\cup\{\infty\}$, if its pullback to $U\times \beta G$ by $\varphi^{-1}$ is of class $C^k$. In turn, a continuous $X$-valued function on $U\times \beta G$ is of class $C^k$ if regarded as a Banach-valued map $U\to C(\beta G;X)$ it has continuous derivatives of order $\le k$ (in local coordinates on $U$).

For a rectangular set $R_{U,H}\subset E(S,\beta G)|_U$ with clopen $H$ a function $f$ on $R_{U,H}$ is said to belong to the space $C^k(R_{U,H};X)$ if its extension to $E(S,\beta G)|_U$ by $0$ belongs to $C^k(E(S,\beta G)|_U;X)$. 

For an open $V\subset E(S;\beta G)$ a continuous function $f$ on $V$ belongs to the space $C^k(V;X)$ if its restriction to each
$R_{U,H}\subset V$  with $H$ clopen belongs to $C^k(R_{U,H};X)$. 

In the proofs we use the following results.
\begin{Proposition}[\cite{Br2}, Proposition 3.2]\label{partition}
For a finite open cover of $E(S;\beta G)$ there exists a $C^\infty$ partition of unity subordinate to it.
\end{Proposition}
\begin{C}[\cite{Br2}, Corollary 3.4]\label{cutoff1}
Let $U\Subset V\subset E(S,\beta G)$ be open. Then there exists a nonnegative $C^\infty$ function $\rho$ on $E(S,\beta G)$ such that $\rho=1$ in an open neighbourhood of $\bar U$ and ${\rm supp}\,\rho\subset V$.
\end{C}

An $X$-valued $(0,1)$-form $\omega$ of class $C^k$ on an open $U\subset E(S,\beta G)$ is defined in each coordinate chart $R_{V,H}\subset U$ with local coordinates $(z,\xi)$ (pulled back from $V\times\beta G$ by $\varphi$) by the formula $\omega|_{R_{V,H}}:=f(z,\xi)d\bar z$, $f\in C^k(R_{V, H};X)$, so that the restriction of the family $\{\omega|_{R_{V,H}}\,;\,R_{V,H}\subset U\}$ to $\Di$ determines a global $X$-valued $(0,1)$-form of class $C^k$ on the open set $U\cap\Di\subset\Di$.  

By $\mathcal E ^{0,1}(U; X)$ we denote the space of $X$-valued $(0,1)$-forms on $U\subset E(S,\beta G)$. The operator $\bar\partial: C^\infty(U;X)\to \mathcal E ^{0,1}(U; X)$ is defined in each $R_{V,H}\subset U$ equipped with the local coordinates $(z,\xi)$ as $\bar\partial f(z,\xi):=\frac{\partial f}{\partial\bar z}(z,\xi)d\bar z$. Then the composite of the restriction map to $U\cap\Di$ with this operator coincides with the standard $\bar\partial$ operator defined on $C^\infty(U\cap\Di;X)$.

It is easy to check, using Cauchy estimates for derivatives of families of uniformly bounded holomorphic functions on $\Di$, that if $f\in {\mathcal O}(U; X)$, $U\subset E(S,\beta G)$, then $f\in C^\infty(U; X)$ and in each $R_{V,H}\subset U$ with local coordinates $(z,\xi)$ the function $f(z,\xi)$ is holomorphic in $z$. Thus $\bar\partial f=0$.

By $\mathcal E_{\rm comp}^{0,1}(M_a; X)$ we denote the class of $X$-valued $C^\infty$ (0,1)-forms on $E(S,\beta G)$ with compact supports in $M_a:=E(S,G_{in})$, i.e., $\omega\in \mathcal E_{\rm comp}^{0,1}(M_a; X)$ if there is a compact subset of $M_a$ such that in each local coordinate representation $\omega=f d\bar z$ support of $f$ belongs to it. By ${\rm supp}\,\omega$ we denote the minimal set satisfying this property. Let $\mathcal E_K^{0,1}(X)\subset \mathcal E_{\rm comp}^{0,1}(M_a; X)$ be the subspace of forms with supports in the compact set
$K\Subset M_a$.

\begin{Th}[\cite{Br2}, Theorem 3.5]\label{dbar}
There exist a norm $\|\cdot\|_{K;X}$ on $\mathcal E_K^{0,1}(X)$ and a continuous linear operator $L_{K;X}: \bigl(\mathcal E_K^{0,1}(X),\|\cdot\|_{K;X}\bigr)\to \bigl(C(M(H^\infty);X),\sup_{M(H^\infty)}\|\cdot\|_X\bigr)$ with norm bounded by a constant depending only on $K$ such that
for each $\omega\in \mathcal E_{K}^{0,1}(X)$ 
\begin{itemize}
\item[(a)]
$L_{K;X}(\omega)|_{M_a}\in C^\infty(M_a;X)$\quad and\quad $\bar\partial(L_{K;X}(\omega)|_{M_a})=\omega$;
\item[(b)]
$L_{K;X}(\omega)|_{M(H^\infty)\setminus K}\in\mathcal O(M(H^\infty)\setminus K;X)$.
\end{itemize}
\end{Th}
Moreover, if $\omega|_\Di=F(z)d\bar z$ for $\omega\in \mathcal E_K^{0,1}(X)$, then
\begin{equation}\label{equiv}
\|\omega\|_{K;X}\le C_K\cdot\sup_{z\in \Di\cap {\rm supp}\,\omega}\left\{\|F(z)\|_X\cdot (1-|z|)\right\}
\end{equation}
for a constant $C_K\ge 1$ depending only on $K$, and if $H$ is a $C^\infty$ function on a neighbourhood of ${\rm supp}\,\omega$ with values in the space $L(X;Y)$ of bounded linear operators between complex Banach spaces $X$ and $Y$, then
\begin{equation}\label{oper}
\|H(\omega)\|_{K;Y}\le \left\{\sup_{z\in {\rm supp}\,\omega}\|H(z)\|_{L(X,Y)}\right\}\cdot\|\omega\|_K ;
\end{equation}
here $H(\omega)|_{\Di}:=H(z)(F(z))d\bar z$.
\sect{Cousin-type Lemma}
Let $X$ be a complex Banach space.
A continuous $X$-valued function on a compact subset $K\subset M(H^\infty)$ is called holomorphic if it is the restriction of a holomorphic function defined in an  open neighbourhood of $K$. The space of such functions is denoted by $\mathcal O(K;X)$. By $A(K;X)$ we denote the closure 
of $\mathcal O(K;X)$ in the Banach space $C(K;X)$ equipped with norm $\|f\|_{K;X}:=\sup_{z\in K}\|f(z)\|_X$. 

\begin{Th}\label{cousin}
Suppose that $U_1, U_2\subset M(H^\infty)$ are open such that $\bar U_1\cap \bar U_2\subset M_a$ and $W_i\Subset U_i$, $i=1,2$, are compact. We set $W:=W_1\cup W_2$. There exists a constant $C$ such that for every function $f\in A(\bar U_1\cap \bar U_2\cap W; X)$ there are functions $f_i\in A(\bar U_i\cap W; X)$ such that
\begin{itemize}
\item[(1)]
$$
f_1+f_2=f\quad\text{on}\quad \bar U_1\cap \bar U_2\cap W ;
$$
\item[(2)]
$$
\|f_i\|_{\bar U_i\cap W; X}\le C\|f\|_{\bar U_1\cap \bar U_2\cap W; X}\, ,\quad i=1,2.
$$
\end{itemize} 
\end{Th} 
\begin{proof}
We consider $M_a$ as an open subset of $E(S,\beta G)$. Let $\rho_i$, $i=1,2$, be nonnegative $C^\infty$ functions on $E(S,\beta G)$ such that
$\rho_i|_{W_i}>0$ and ${\rm supp}\,\rho_i\Subset U_i$, cf. Corollary \ref{cutoff1}. Let $Z:=\{x\in E(S,\beta G)\, ;\, \rho_1(x)=\rho_2(x)=0\}$. By definition, $Z\cap W=\emptyset$. Therefore there is an open neighbourhood $U$ of $W$ such that $\bar U\cap Z=\emptyset$. Then we define nonnegative $C^\infty$ functions $\psi_i$ on $U$ by the formulas
\[
\psi_i:=\frac{\rho_i}{\rho_1+\rho_2},\quad i=1,2;
\]
here ${\rm supp}\,\psi_i\subset U_i\cap U$. 

Without loss of generality we may assume that $\emptyset\ne\bar U_1\cap \bar U_2\cap W\ne W$. (For otherwise, the statement of the theorem is trivial.)
We fix an open neighbourhood $N\Subset M_a$ of $\bar U_1\cap \bar U_2\cap W$.

By definition, a function $f\in \mathcal O(\bar U_1\cap \bar U_2\cap W; X)$ admits a holomorphic extension (denoted also by $f$) to an open set $O_f\Subset N$ such that
$O_f\supset \bar U_1\cap \bar U_2\cap W$ and $\|f\|_{\bar O_f ;X}\le 2\|f\|_{\bar U_1\cap \bar U_2\cap W; X}$. Without loss of generality we may assume that $O_f\subset N\cap U$. Next, we can find open sets $U_i'\supset\bar U_i$, $W_i'\Subset U_i'\cap U$ such that $W_i\subset W_i'$, $i=1,2$, and $W'\ne U_1'\cap U_2'\cap W'\Subset O_f$; here $W':=W_1'\cup W_2'$. We set
\[
f_1':=\psi_2 f\quad\text{on}\quad U_1'\cap W',\qquad f_2':=-\psi_1 f\quad\text{on}\quad U_2'\cap W'.
\]
By the definition these functions are well defined. We check it, say, for $f_1'$. Indeed, by our construction $f_1'$ is defined in a neighbourhood of the closure of $U_1'\cap U_2'\cap W'$. Since 
\[
(U_1'\cap W')\setminus (U_1'\cap U_2'\cap W')=(U_1'\setminus U_2')\cap W'\subset (U_1'\cap U)\setminus (U_2\cap U),
\]
$\psi_2=0$ on $(U_1'\cap W')\setminus (U_1'\cap U_2'\cap W')$. Hence, $f_1'=0$ on $(U_1'\cap W')\setminus (U_1'\cap U_2'\cap W')$ as well. Similarly, one proves that $f_i'$ is of class $C^\infty$ on $U_i'\cap W'$, $i=1,2$. Thus,
\[
\omega_f:=\bar\partial f_i'\qquad \text{on}\quad U_i'\cap W',\quad i=1,2,
\]
is an $X$-valued $C^\infty$ $(0,1)$ form on $W'$ (because $f_1'-f_2'=f$ on $U_1'\cap U_2'\cap W'$).
By the definition ${\rm supp}\,\omega_f\subset O_f$. Let $\rho_f$ be a $C^\infty$ function on $E(S,\beta G)$  with values in $[0,1]$ equals $1$ in an open neighbourhood of $W$ with support in $W'$, cf. Corollary \ref{cutoff1}. Then  $\eta_f:=\rho_f\omega_f$ is an $X$-valued $C^\infty$ $(0,1)$ form on $E(S,\beta G)$ with support in $O_f\Subset N$. 

Let us apply this construction to $X:=\Co$ and $f(z):=I$, $z\in \bar U_1\cap \bar U_2\cap W$, where $I:=1$ is the unit of $\Co$. In this case $\rho_I$ and $\omega_I$ are some specific function and form depending only on the choice of the above sets for $I$. Therefore for our space $X$ and function $f$ we may choose $\rho_f$ so that $\rho_I=1$ on ${\rm supp}\, \rho_f$. Since ${\rm supp}\,\eta_f$ is a compact subset of $O_f$, there is a $C^\infty$ function $h_f$ on $E(S,\beta G)$ with values in $[0,1]$ equals $1$ in an open neighbourhood of ${\rm supp}\,\eta_f$ with support in $O_f$. In particular, we obtain
\[
\eta_f=h_f\cdot \chi_{O_f}\cdot f|_{O_f}\cdot\rho_f\cdot\eta_I,
\]
where $\chi_{O_f}$ is the characteristic function of $O_f$. The $X$-valued function $h_f\cdot \chi_{O_f}\cdot f|_{O_f}\cdot\rho_f$ is of class $C^\infty$ on $E(S,\beta G)$. Therefore applying estimate \eqref{oper} we get
$$
\begin{array}{l}
\displaystyle
\|\eta_f\|_{N;X}\le \left\{\sup_{z\in O_f}\|h_f(z)\cdot \chi_{O_f}(z)\cdot f(z)\cdot\rho_f(z)\|_X\right\}\cdot\|\eta_I\|_{N;\Co}\\
\\
\qquad\qquad\le (2\|\eta_I\|_{N;\Co})\|f\|_{\bar U_1\cap \bar U_2\cap W; X}.
\end{array}
$$
Now by Theorem \ref{dbar} we find an $X$-valued continuous function $g$ on $M(H^\infty)$ such that $\bar\partial g=\eta_f$ on $M_a$ and 
$\|g\|_{M(H^\infty);X}\le A_N\|\eta_f\|_{N;X}\le A'\|f\|_{\bar U_1\cap \bar U_2\cap W; X}$ with $A':= 2A_N\|\eta_I\|_{N;\Co}$. 

Finally, since $\eta_f=\omega_f$ in an open neighbourhood of $W$, defining
$$
f_i:=(-1)^{i-1}(f_i'-g)\quad\text{on}\quad \bar U_i\cap W
$$
we obtain $f_1+f_2=f $ on $\bar U_1\cap\bar U_2\cap W$ and $\|f_i\|_{\bar U_i\cap W; X}\le C\|f\|_{\bar U_1\cap \bar U_2\cap W; X}$, $i=1,2$, for some $C$ independent of $f$. By our construction, $f_i\in\mathcal O(\bar U_i\cap W;X)$, $i=1,2$.

This proves the result for functions from $\mathcal O(\bar U_1\cap \bar U_2\cap W; X)$.

If now, $f\in A(\bar U_1\cap \bar U_2\cap W; X)\setminus\mathcal O(\bar U_1\cap \bar U_2\cap W; X)$, choose a sequence $\{g_n\}_{n\in\N_+}\subset \mathcal O(\bar U_1\cap \bar U_2\cap W; X)$, $g_0:=0$, converging to $f$ and such that 
\[
\|g_{n}-g_{n-1}\|_{\bar U_1\cap \bar U_2\cap W; X}\le \frac{\|f\|_{\bar U_1\cap \bar U_2\cap W; X}}{2^{n-1}}\quad\text{for all}\quad n\in\N.
\]
We set $h_{n}:=g_{n}-g_{n-1}$, $n\in\N$. By the above result, there exist functions $h_{in}\in A( \bar U_i\cap W; X)$ such that $h_{1n}+h_{2n}=h_n$ on $\bar U_1\cap \bar U_2\cap W$ and $\|h_{in}\|_{\bar U_i\cap W; X}\le C\|h_n\|_{\bar U_1\cap \bar U_2\cap W; X}\le\frac{C}{2^{n-1}}\|f\|_{\bar U_1\cap \bar U_2\cap W; X}$, $i=1,2$. We set
\[
f_i:=\sum_{n=1}^\infty h_{in},\quad\text{on}\quad \bar U_i\cap W,\quad i=1,2.
\]
The above estimates imply that the series converge to elements of $A(\bar U_i\cap W;X)$, $i=1,2$. Moreover,
\[
\|f_i\|_{\bar U_i\cap W; X}\le 3C\|f\|_{\bar U_1\cap \bar U_2\cap W; X},\quad i=1,2,
\]
and 
\[
f_1+f_2=\sum_{n=1}^\infty (h_{1n}+h_{2n})=\sum_{n=1}^\infty (g_{n}-g_{n-1})=f\quad\text{on}\quad \bar U_1\cap \bar U_2\cap W.
\]

The proof of the theorem is complete.
\end{proof}
\sect{Cartan-type Lemma}
For the basic facts of the the theory of Banach Lie groups see, e.g., \cite{Mai}.

Let $B^{-1}$ be the group of invertible elements of a complex unital Banach algebra $B$. Let $\psi: G\to B_0^{-1}$ be a regular covering of the connected component $B_0^{-1}$ of $B^{-1}$ containing the unit $1_{B^{-1}}$ of $B^{-1}$. Then $G$ is complex Banach Lie group and $\psi$ is a surjective morphism of complex Lie groups whose kernel is a discrete central subgroup of $G$, see, e.g., \cite{Po}.
By $\exp_G:\mathcal L_G\to G$, $\exp_{B^{-1}}: \mathcal L_{B^{-1}}\to B^{-1}$ we denote the corresponding exponential maps of Lie algebras of $G$ and $B^{-1}$. Then differential $d\psi(1_G):\mathcal L_G\to \mathcal L_{B^{-1}}$ at the unit $1_G$ of $G$ is an isomorphism of Lie algebras and 
\begin{equation}\label{comm}
\psi\circ\exp_{G}=\exp_{B^{-1}}\circ\, d\psi(1_G).
\end{equation}

A continuous map $f: W\to G$, where $W\subset M(H^\infty)$ is open, is called holomorphic if $f|_{\Di\,\cap W}:\Di\cap W\to G$ is a holomorphic map of complex Banach manifolds. For a compact set $K\subset M(H^\infty)$ a map $K\to G$ is called holomorphic if it is a restriction of a holomorphic map into $G$ defined on an open neighbourhood of $K$. The space of such maps is denoted by $\mathcal O(G;X)$. We say that a continuous map $f: K\to G$ belongs to the space $A(K;G)$ if $\psi(f|_S)\in A(S;B)$ for every compact subset $S\subset K$. The space $A(K;G)$ has a natural group structure induced by the product on $G$. Clearly, $\mathcal O(K;G)\subset A(K;G)$ (because $\psi$ is a holomorphic map).
The unit of $A(K;G)$, i.e., the map $z\mapsto 1_{G}$, $z\in K$, will be denoted by $I$. 

A map $\gamma : [0,1]\to A(K; G)$ is called a {\em path} if the induced map $[0,1]\times K\to G$, $(t, x)\mapsto \gamma(t)(x)$, is continuous. The set of all maps in $A(K; G)$ that can be joined by paths (in $A(K; G)$) with $I$ will be called the {\em connected component} of $I$.

We retain notation of Theorem \ref{cousin}.
\begin{Th}\label{cartan}
Assume that $W_1$ is holomorphically convex and $W_1\cap W_2\ne\emptyset$.
Let $F\in A(\bar U_1\cap \bar U_2;G)$ belong to the connected component of $I$. Then there exist 
$F_i\in A(W_i; G)$ such that $F_1\cdot F_2=F$ on $W_1\cap W_2$.
\end{Th}
\begin{proof}
Since $W_1$ is holomorphically convex, according to Proposition \ref{holconv}, there exists a sequence of open sets $U_{1,i}\Subset U_1$ containing $W_1$ such that  $U_{1,i+1}\Subset U_{1,i}$ and $\bar U_{1, i}$ is holomorphically convex for all $i\in\mathbb N$. Let us choose a sequence of open 
sets $U_{2,i}\Subset U_2$ containing $W_2$ such that $U_{2,i+1}\Subset U_{2,i}$ for all $i\in\N$.
Then Theorem \ref{cousin} is also valid with $U_j$ replaced by $U_{j,i}$, $j=1,2$,  $W:=W_1\cup W_2$ replaced by $W^i:=\bar U_{1,i+1}\cup\bar U_{2,i+1}$ and $C$ replaced by some $C_i$.

\begin{Proposition}\label{small}
There exists a positive number $\varepsilon_i$ such that for every function $F\in A(\bar U_{1,i}\cap\bar U_{2,i}\cap W^i; G)$ satisfying $\|\psi\circ F-\psi\circ I\|_{\bar U_{1,i}\cap \bar U_{2,i}\cap  W^i; B}<\varepsilon_i$ there exist $F_j\in A(\bar U_{j,i}\cap W^i; G)$, $j=1,2$, such that 
\begin{itemize}
\item[(1)]
$F_{1}\cdot F_{2}=F$ on $\bar U_{1,i}\cap\bar U_{2,i}\cap W^i$ and 
\item[(2)]
$\|\psi\circ F_j-\psi\circ I\|_{\bar U_{j,i}\cap W^i; B}\le 4C_i\|\psi\circ F-\psi\circ I\|_{\bar U_{1,i}\cap\bar U_{2,i}\cap W^i; B}$, $j=1,2$.
\end{itemize}
Moreover, $\psi$ has an inverse $\psi^{-1}$ on the ball $\{v\in B\, ;\, \|v-1_{B^{-1}}\|_B<4C_i\varepsilon_i\}\subset B$ such that $\psi^{-1}(1_{B^{-1}})=1_G$.
\end{Proposition}
\begin{proof}
Since $G$ is locally isomorphic to $B^{-1}$, cf. \eqref{comm}, it suffices to prove the result for $G:=B^{-1}$. 
In this case the proof repeats literally the classical proof of Cartan's lemma with the classical Cousin lemma replaced by Theorem \ref{cousin} and with matrix norm replaced by norm on $B$, see, e.g., \cite[Ch.~III.1]{GR}.
\end{proof}
Further, since $F$ in the statement of the theorem belongs to the connected component $\mathcal C_0$ containing $I$ of the group $A(\bar U_1\cap \bar U_2; G)$, there exists a path $\gamma : [0,1]\to \mathcal C_0$ such that $\gamma(0)=I$ and $\gamma(1)=F$. From the continuity of $\gamma$ follows that there is a partition $0=t_0<t_1<\cdots <t_k=1$ of $[0,1]$ such that 
\begin{equation}\label{eq4.1}
\|\psi\circ(\gamma(t_i)^{-1}\cdot\gamma(t_{i+1}))-\psi\circ I\|_{\bar U_{1,1}\cap \bar U_{2,1}\cap W^1; B}<\varepsilon_1\quad\text{ for all}\quad i.
\end{equation}
Applying Proposition \ref{small} to each $\gamma(t_i)^{-1}\cdot\gamma(t_{i+1})$ we obtain
\begin{equation}\label{eq4.2}
\gamma(t_i)^{-1}\cdot\gamma(t_{i+1})=F_{1}^i\cdot F_{2}^i\quad\text{on}\quad \bar U_{1,1}\cap \bar U_{2,1}\cap W^1
\end{equation}
for some $F_{\ell}^i\in A(\bar U_{\ell,1}\cap W^1; G)$, $\ell=1,2$.

Next, we have
\[
F=\prod_{i=0}^{k-1}(\gamma(t_i)^{-1}\cdot\gamma(t_{i+1})).
\]

We will use induction on $j$ for $1\le j\le k$. The induction hypothesis is {\em if
\[
F^j:=\prod_{i=0}^{j-1}(\gamma(t_i)^{-1}\cdot\gamma(t_{i+1})),
\]
then 
\[
F^j=F_{1,j}\cdot F_{2,j}\quad\text{on}\quad \bar U_{1,j}\cap \bar U_{2,j}\cap W^j
\]
for some $F_{\ell,j}\in A(\bar U_{\ell,j}\cap W^j; G)$, $\ell=1,2$. }

By \eqref{eq4.2} the statement is true for $j=1$. Assuming that it is true for $j-1$ let us prove it for $j$.

By the induction hypothesis we have
\[
F^j=F^{j-1}\cdot \gamma(t_j)^{-1}\cdot\gamma(t_{j+1})= F_{1,j-1}\cdot F_{2,j-1}\cdot F_{1}^{j-1}\cdot F_{2}^{j-1}\quad\text{on}\quad
\bar U_{1,j-1}\cap\bar U_{2,j-1}\cap W^{j-1}.
\]
Since $\bar U_{1,j}\subset \bar U_{1,j-1}\cap W^{j-1}$ is holomorphically convex, by Theorem \ref{runge}, assuming that $\varepsilon_1$ is sufficiently small such that $\exp_{B^{-1}}^{-1}((\psi\circ F_{1}^{j-1}))$ is well-defined on $\bar U_{1,j}$ (see Proposition \ref{small}\,(2)), we approximate
$\exp_{B^{-1}}^{-1}((\psi\circ F_{1}^{j-1}))$ uniformly on $\bar U_{1,j}$ by a sequence of functions from $\mathcal O(M(H^\infty);\mathcal L_{B^{-1}})$. Applying to the functions of this sequence map $\exp_{G}\circ (d\psi(1_G))^{-1}$, see \eqref{comm}, we obtain that for every $\varepsilon>0$ there exists
$F_{\varepsilon}\in \mathcal O(M(H^\infty);G)$ such that
\begin{equation}\label{eq4.3}
\|\psi(F_{1}^{j-1}\cdot F_{\varepsilon}^{-1})-\psi(I)\|_{\bar U_{1,j};B}<\varepsilon .
\end{equation}
We write
\[
F_{2,j-1}\cdot F_{1}^{j-1}\cdot F_{2}^{j-1}=[F_{2,j-1},F_{1}^{j-1}\cdot F_{\varepsilon}^{-1}]\cdot (F_{1}^{j-1}\cdot F_{\varepsilon}^{-1})\cdot F_{2,j-1}\cdot F_{\varepsilon}\cdot  F_{2}^{j-1},
\]
where $[A_1,A_2]:=A_1\cdot A_2\cdot A_1^{-1}\cdot A_2^{-1}$.

According to \eqref{eq4.3} for a sufficiently small $\varepsilon$ we have
\[
\|\psi([F_{2,j-1},F_{1}^{j-1}\cdot F_{\varepsilon}^{-1}]\cdot (F_{1}^{j-1}\cdot F_{\varepsilon}^{-1}))-\psi(I)\|_{\bar U_{1,j}\cap \bar U_{2,j}\cap W^j ; B}<\varepsilon_j .
\]
Thus by Proposition \ref{small}, there exist $H_{\ell}\in A(\bar U_{\ell,j}\cap W^j; G)$, $\ell=1,2$, such that 
\[
[F_{2,j-1},F_{1}^{j-1}\cdot F_{\varepsilon}^{-1}]\cdot (F_{1}^{j-1}\cdot F_{\varepsilon}^{-1})=H_1\cdot H_2\quad\text{on}\quad \bar U_{1,j}\cap \bar U_{2,j}\cap W^j.
\]
In particular, we obtain
\[
F^j=F_{1,j-1}\cdot F_{2,j-1}\cdot F_{1}^{j-1}\cdot F_{2}^{j-1}=F_{1,j-1}\cdot H_1\cdot H_2\cdot F_{2,j-1}\cdot F_{\varepsilon}\cdot  F_{2}^{j-1}\quad\text{on}\quad \bar U_{1,j}\cap \bar U_{2,j}\cap W^j.
\]
We set
\[
F_{1,j}:=F_{1,j-1}\cdot H_1,\quad\text{on}\quad \bar U_{1,j}\cap W^j,\qquad F_{2,j}:=H_2\cdot F_{2,j-1}\cdot F_{\varepsilon}\cdot  F_{2}^{j-1}\quad\text{on}\quad \bar U_{2,j}\cap W^j.
\]
This completes the proof of the induction step.

From this result for $j:=k-1$ and the fact that $W_\ell\subset \bar U_{\ell,k+1}$, $\ell=1,2$, we obtain that there exist 
$F_i\in A(W_i; G)$ such that $F_1\cdot F_2=F$ on $W_1\cap W_2$.
\end{proof}
\sect{Maximal Ideal Space of Algebra $A(K)$}
We set $A(K):=A(K;\Co)$, where $K\subset M_a$ is compact. By definition, $A(K)$ is a complex commutative unital Banach algebra with norm
$\|f\|_\infty:=\sup_{z\in K}|f(z)|$.
\begin{Th}\label{ideal}
The maximal ideal space of $A(K)$ is $K$.
\end{Th}
\begin{proof}
Since $A(K)$ is the closure of $\cup_{U\supset K}\left(\mathcal O(U)|_K\right)$, where $U$ runs over all possible open neighbourhoods of $K$, to prove the result it suffices to establish the following (cf. \cite[Ch. V, Th. 1.8]{Ga}):
\begin{Lm}\label{le5.2}
Suppose that $f_1,\dots, f_n\in \mathcal O(U)$ for an open $U$, $K\subset U\Subset M_a$, satisfy 
\begin{equation}\label{corona}
\sum_{j=1}^n|f_j(z)|>\delta>0\quad\text{for all}\quad z\in U.
\end{equation}
Then there exist $g_1,\dots, g_n\in A(K)$ such that
\begin{equation}\label{bezout}
\sum_{j=1}^n g_jf_j=1\quad\text{on}\quad K.
\end{equation}
\end{Lm}
\begin{proof}
Using \eqref{corona} we find open sets $W_k\Subset U$, $1\le k\le\ell$, and bounded functions $g_{ik}\in\mathcal O(W_k)$ such that
$K\subset \cup_{k=1}^\ell W_k$ and for all $k$
\begin{equation}\label{eq3.5'}
\sum_{i=1}^n g_{ik}f_i=1\quad\text{on}\quad W_k.
\end{equation}
Let $\{\varphi_k\}_{1\le k\le \ell}$ be nonnegative $C^\infty$ functions on $E(S,\beta G)$ such that ${\rm supp}\,\varphi_k\Subset W_k$ and $\sum_{k=1}^\ell\varphi_k=1$ in an open neighbourhood $O\Subset M_a$ of $K$ (see subsection 2.3).
We set 
\[
c_{i,rs}:=g_{ir}-g_{is}\quad\text{on}\quad W_r\cap W_s\ne\emptyset ,
\]
and 
\[
h_{ir}:=\sum_{s} \varphi_s c_{i,rs},
\]
where the sum is taken over all $s$ for which $W_s\cap W_r\ne\emptyset$.
Then $h_{ir}\in C^\infty(W_r)$ and
\[
h_{ir}-h_{is}=c_{i,rs}\quad\text{on}\quad O\cap W_r\cap W_s\ne\emptyset.
\]
Next, we define
\[
h_i:=g_{ir}-h_{ir}\quad\text{on}\quad O\cap W_r .
\]
Clearly, each $h_i\in C^\infty(O)$.  In particular, $\bar\partial h_i$ are $C^\infty$ $(0,1)$ forms on $O$.
Consider equations
\[
\bar\partial b_{is}=h_i\bar\partial h_s\quad\text{on}\quad O.
\]
According to \cite[Lm.~4.2]{Br2}, for an open set $O'$ such that $K\subset O'\Subset O$ there are $b_{is}\in C^\infty(O')$ solving these equations on $O'$.
We set
\[
g_i:=h_{i}+\sum_{s=1}^n (b_{is}-b_{si})f_s\quad\text{on}\quad O'.
\]
Since by \eqref{eq3.5'} $\sum_{i=1}^n h_i f_i=1$ on $O$,
\[
\sum_{i=1}^n g_if_i=1\quad\text{on}\quad O'\quad\text{and}
\]
\[
\bar\partial g_i=\bar\partial h_i+\sum_{s=1}^n f_s\cdot (h_i\bar\partial h_s- h_s\bar\partial h_i)=\bar\partial h_i+h_i\bar\partial\left(\sum_{s=1}^n f_s h_s\right)-\bar\partial h_i\sum_{s=1}^n f_s h_s=0.
\]
Hence, $g_i\in \mathcal O(O')$ (cf. \cite[Ch.~VIII, Th.~2.1]{Ga} for similar arguments).
\end{proof}
The proof of the theorem is complete.
\end{proof}
\begin{R}\label{re1}
{\rm Theorem \ref{ideal} is valid for any compact $K\subset M(H^\infty)$. This fact is not used in the proof.}
\end{R}
\sect{Triviality of Holomorphic Principal Bundles}
\subsect{Bundles with Simply Connected Fibres}
Let $B$ be a complex unital Banach algebra and $\psi:\widetilde B_0^{-1}\to B_0^{-1}$ be the universal covering of the connected component $B_0^{-1}$ of $B^{-1}$ containing $1_{B^{-1}}$. 
Let $p:E\to M(H^\infty)$ be a holomorphic principal bundle with fibre $\widetilde B_0^{-1}$. It is defined on an open cover $\mathcal U=(U_i)_{i\in I}$ of $M(H^\infty)$ by a cocycle $g=\{g_{ij}\in \mathcal O (U_i\cap U_j; \widetilde B_0^{-1})\}$ similarly to the construction of subsection~1.1 (with $X$ replaced by $\widetilde B_0^{-1}$). For a compact set $K\subset M(H^\infty)$, we say that $E|_K$ is holomorphically trivial if it is holomorphically trivial in an open neighbourhood of $K$ (the notion is defined analogously to that of subsection 1.1).

Any open set of the form
$$
O_{f,\varepsilon}:=\{x\in M(H^\infty)\, ;\, |\hat f(x)|<\varepsilon\},
$$
where $f$ is an interpolating Blaschke product ($\hat f$ is the extension of $f$ to $M(H^\infty)$ by means of the Gelfand transform) and
$\epsilon$ is so small that $O_{f,\varepsilon}\Subset M_a$ and is homeomorphic to $\Di_\varepsilon\times \hat f^{-1}(0)$, see subsection~2.2.2, will be called an open {\em Blaschke set}.
\begin{Lm}\label{compactif}
Let $N$ be a compact subset of the Blaschke set $O_{f,\varepsilon}$. Then every continuous map $N\to \widetilde B_0^{-1}$ is homotopic to a constant map.
\end{Lm}
\begin{proof}
For the basic facts of algebraic topology see, e.g., \cite{ES}.

Since $f^{-1}(0)$ is an interpolating sequence for $H^\infty$, the set $\hat f^{-1}(0)$ is homeomorphic to the Stone-\v{C}ech compactification $\beta\N$. By the result of \cite{Ma}, $\beta\N$ can be obtained as the (iterated) inverse limit of a family of finite sets. Hence, $O_{f,\varepsilon}$ can be obtained as the (iterated) inverse limit of a family of sets of the form $\Di_{\varepsilon}\times {\mathcal F}$, where ${\mathcal F}$ is a finite set. In turn, $N$ can be obtained as the (iterated) inverse limit of a family of compact subsets $K\subset \Di_{\varepsilon}\times {\mathcal F}$. Since $\widetilde B_0^{-1}$ is an absolute neighbourhood retract, see \cite{P}, any continuous map $N\to \widetilde B_0^{-1}$ is homotopic to a continuous map into $\widetilde B_0^{-1}$ pulled back from some $K\subset \Di_{\varepsilon}\times {\mathcal F}$ in the inverse limit construction, see, e.g., an argument in the proof of Theorem 1.2 of \cite{BRS}. Thus it suffices to prove that every continuous map $K\to \widetilde B_0^{-1}$ with such a $K$ is homotopic to a constant map.
Further, since $\widetilde B_0^{-1}$ is an absolute neighbourhood retract, every continuous map $K\to \widetilde B_0^{-1}$, can be extended to an open neighbourhood $W\subset \Di_{\varepsilon}\times {\mathcal F}$ with a smooth boundary. In particular, $W$ has finitely many connected components homeomorphic to multiply connected domains in $\Co$. This implies that it suffices to establish the result for continuous maps into $\widetilde B_0^{-1}$ defined on a multiply connected domain $D$. Finally, there exists a finite connected one-dimensional $CW$ complex $S\subset D$ such that $S$ is a strong deformation retract of $D$. Thus every continuous map $D\to \widetilde B_0^{-1}$ is homotopic to the pullback under the retraction of a continuous map $S\to \widetilde B_0^{-1}$. Since $\widetilde B_0^{-1}$ is simply connected, every continuous map $S\to \widetilde B_0^{-1}$ is homotopic to a constant map.
\end{proof}

\begin{Proposition}\label{sew}
Suppose $W_i\Subset U_i\subset M(H^\infty)$, $i=1,2$, are open such that $\bar W_1$ is holomorphically convex and $\bar U_1\cap \bar U_2$ is a subset of an open Blaschke set $O_{B,\varepsilon}$. We set $W:=W_1\cup W_2$. If $E|_{\bar U_i}$, $i=1,2$, are holomorphically trivial, then $E|_W$ is holomorphically trivial.
\end{Proposition}
\begin{proof}
Without loss of generality we may assume that $\bar W_1\cap\bar W_2$ is nonempty and $\bar U_1\cap\bar U_2$ is proper in each $\bar U_i$, $i=1,2$ (for otherwise the statement is trivial). Then, by the hypothesis, $E|_{\bar U_1\cup\bar U_2}$ is determined by a cocycle $c\in \mathcal O(\bar U_1\cap\bar U_2; \widetilde B_0^{-1})$. 
Let $\mathcal C$ be a connected component of the Banach group $A(\bar U_1\cap\bar U_2; B^{-1})$ containing $\psi\circ c$.
According to \cite[Th.~1.19]{Br2} $A(\bar U_1\cap\bar U_2; B)$ is the completion of the algebraic tensor product $A(\bar U_1\cap\bar U_2)\otimes B$ with respect to norm 
\[
\left\|\sum_{k=1}^m a_k\otimes b_k\right\|:=\sup_{x\in \bar U_1\cap\bar U_2}\left\|\sum_{k=1}^m a_k(x) b_k\right\|_{B}\quad \text{with}\quad a_k\in A(\bar U_1\cap\bar U_2),\, b_k\in B.
\]
Therefore (because $\mathcal C$ is open) there exists $\tilde c\in \mathcal C\cap\bigl(A(\bar U_1\cap\bar U_2)\otimes B\bigr)$ sufficiently close to $\psi\circ c$ such that $t(\psi\circ c)+(1-t)\tilde c\in\mathcal C$ for all $t\in [0,1]$. In turn, $A(\bar U_1\cap\bar U_2)\otimes B$ belongs to the projective tensor product $A(\bar U_1\cap\bar U_2)\widehat{\otimes} B$. By the result of \cite{D} the set of connected components of the group of invertible elements of the Banach algebra $A(\bar U_1\cap\bar U_2)\widehat{\otimes} B$ is in a one-to-one correspondence (defined by the Gelfand transform on $A(\bar U_1\cap\bar U_2)$) with the set of homotopy classes of continuous maps $M(A(\bar U_1\cap\bar U_2))\to B^{-1}$. 
By Theorem \ref{ideal}, $M(A(\bar U_1\cap\bar U_2))=\bar U_1\cap\bar U_2\subset O_{B,\varepsilon}$. Then by Lemma \ref{compactif},
the map $\psi\circ c$ is homotopic inside $B^{-1}$ to $I$  (here $I(z):=1_{B^{-1}}$ for all $z$). In turn, the map $\tilde c: \bar U_1\cap\bar U_2\to B^{-1}$ is homotopic to $I$. Thus $\tilde c\in (A(\bar U_1\cap\bar U_2)\widehat{\otimes} B)^{-1}$ belongs to the connected component containing $I$. Observing that $(A(\bar U_1\cap\bar U_2)\widehat{\otimes} B)^{-1}\subset A(\bar U_1\cap\bar U_2; B^{-1})$ we get that $\psi\circ c$ can be joined by a path $\gamma: [0,1]\to A(\bar U_1\cap\bar U_2; B^{-1})$ with $I$. 
Further, since $\psi: \widetilde B_0^{-1}\to B_0^{-1}$ is a principal bundle with discrete fibre and $\bar U_1\cap\bar U_2$ is compact, by the covering homotopy theorem, see, e.g., \cite{Hu},  there exists a unique continuous map $R:[0,1]\times \bar U_1\cap\bar U_2\to \widetilde B_0^{-1}$ such that
\[
\psi(R(t,x))=\gamma(t)(x),\quad t\in [0,1],\quad\text{and}\quad  R(1,\cdot)=c.
\]
Since $\psi$ is locally biholomorphic, $R(t,\cdot)\in A(\bar U_1\cap\bar U_2;\widetilde B_0^{-1})$ for each $t\in [0,1]$. We define a path $\widetilde\gamma: [0,1]\to A(\bar U_1\cap\bar U_2;\widetilde B_0^{-1})$ by the formula $\widetilde\gamma(t):=R(t,\cdot)$, $t\in [0,1]$. By the definition $\widetilde\gamma(0):\bar U_1\cap\bar U_2\to \widetilde B_0^{-1}$ is a continuous map into the discrete group $\psi^{-1}(1_{B^{-1}})$. Since $\bar U_1\cap\bar U_2$ is compact, the image of $\widetilde\gamma(0)$ consists of finitely many points. In particular, there exists a path $\gamma': [0,1]\to A(\bar U_1\cap\bar U_2;\widetilde B_0^{-1})$ which joins $\widetilde\gamma(0)$ with
$I$ (here $I(z):=1_{\widetilde B_0^{-1}}$ for all $z$). The product of $\widetilde\gamma$ and $\gamma'$ gives a path in $A(\bar U_1\cap\bar U_2;\widetilde B_0^{-1})$ joining $c$ with $I$. Hence, we can apply Theorem \ref{cartan} to $c$.  Then we find $c_i\in A(\bar W_i; G)$, $i=1,2$, such that $c_1^{-1}\cdot c_2= c$ on $\bar W_1\cap \bar W_2$. But this shows that $E|_{W}$ is holomorphically trivial, as required.
\end{proof}
\begin{Th}\label{principal}
Let $p:E\to M(H^\infty)$ be a holomorphic principal bundle with fibre $\widetilde B_0^{-1}$. Then $E$ is holomorphically trivial.
\end{Th}
\begin{proof}
Since $M_s$ is totally disconnected, there exists a finite open cover $(V_j)$ of $M_s$ such that $\bar V_{i}\cap \bar V_{j}=\emptyset$ for all $i\ne j$ and $E|_{V_j}$ is trivial for all $j$. We set $U_1:=\cup_j V_j$. Then $E|_{U_1}$ is trivial. Since $M(H^\infty)\setminus U_1$ is a compact subset of $M_a$, we can cover it by open Blaschke sets $U_2,\dots , U_N$ such that $\bar U_j\Subset M_a$ and $E|_{U_j}$ is trivial for all $2\le j\le N$.
Next, we can find a refinement $(W_j)_{j\in J}$ of the cover $(U_i)$ such that $W_1\Subset U_1$ and all $\bar W_j$, $j\ne 1$, are holomorphically convex compact subsets of $M_a$, see subsection~2.2.2.
Choosing an open neighbourhood of each $\bar W_j$ containing in $U_{\tau(j)}$ ($\tau$ is the refinement map), without loss of generality we may assume that the sets of indices of these covers coincide, i.e., $\bar W_j\subset U_j$ for all $1\le j\le N$.

Further, according to Proposition \ref{holconv} we can find a sequence of open sets $W_{ij}$, $i\in\N$, $1\le j\le N$, such that $\bar W_j\subset W_{ij}$ for all $i$,
each $W_{ij}\Subset U_j$ and $\bar W_{ij}$, $2\le j\le N$, are holomorphically convex. We set $\mathcal W_i:=(W_{ij})_{1\le j\le N}$. Then $\mathcal W_i$ are finite open covers of $M(H^\infty)$ such that all nonempty $W_{ij_1}\cap W_{ij_2}$ with $j_1\ne j_2$ are relatively compact in $M_a$ (and belong to open Blaschke sets).  For $1\le k\le N$ we set
\[
Z_{ik}:=\bigcup_{j=1}^k W_{ij}.
\]
Using induction on $k$ we prove that if $E|_{Z_{k-1\, k-1}}$ is holomorphically trivial in the category of $\widetilde B_0^{-1}$-bundles, then $E|_{Z_{k k}}$ is holomorphically trivial in this category as well.

By our definition $E|_{Z_{1 1}}$ is holomorphically trivial in the category of $\widetilde B_0^{-1}$-bundles. Provided that the statement is valid for $k-1$ let us prove it for $k$.

To this end in Proposition \ref{sew} consider $U_1:=W_{k-1\, k}$, $U_2:=Z_{k-1\, k-1}$, $W_1:=W_{k k}$, $W_2:=Z_{k k}$. Then the required statement follows from the proposition.

For $k:=N$ we obtain $Z_{N N}=M(H^\infty)$. Thus $E$ is holomorphically trivial as required. 
\end{proof}
\subsect{Bundles with Connected Fibres}
\begin{Th}\label{principal1}
Let $p:E\to M(H^\infty)$ be a holomorphic principal bundle with fibre $B_0^{-1}$. Then $E$ is holomorphically trivial.
\end{Th}

\begin{proof} Suppose that $E$ is defined on a finite open cover $\mathcal U=(U_i)_{i\in I}$ of $M(H^\infty)$ by a cocycle $g=\{g_{ij}\in \mathcal O (U_i\cap U_j; B_0^{-1})\}$. Passing to a suitable refinement, if necessary, without loss of generality we may assume that for each $i, j\in I$ there exists an open simply connected set $W_{ij}\subset B_0^{-1}$ such that $g_{ij}(U_i\cap U_j)\Subset W_{ij}$ and each $g_{ij}$ admits a continuous extension to $\overline{U_i\cap U_{j}}$. Since $W_{ij}$ is simply connected, an inverse $\psi_{ij}:W_{ij}\to \widetilde B_0^{-1}$ of $\psi:\widetilde B_0^{-1}\to B_0$ is defined. Since $g_{ii}\equiv 1_{B_0^{-1}}$ and $g_{ij}=g_{ji}^{-1}$, we can choose $\psi_{ii}$ so that $\psi_{ii}(1_{B_0^{-1}})=1_{\widetilde B_0^{-1}}$ and
$W_{ji}:=\{z\in B_0^{-1}\, ;\, z^{-1}\in W_{ij}\}$, $\psi_{ji}(z):=(\psi_{ij}(z^{-1}))^{-1}$, $z\in W_{ji}$. 

We set $h_{ij}:=\psi_{ij}\circ g_{ij}$. Then $h_{ii}\equiv 1_{\widetilde B_0^{-1}}$, $h_{ji}=h_{ij}^{-1}$ and for pairwise distinct $i,j,k\in I$
\begin{equation}\label{cocycle}
\psi(h_{ij}\cdot h_{jk}\cdot h_{ki})\equiv 1_{B_0^{-1}}\quad\text{on}\quad U_i\cap U_j\cap U_k\ne\emptyset .
\end{equation}

Let $Z\subset \widetilde B_0^{-1}$ be the kernel of $\psi$. Then $Z$ is a discrete abelian subgroup of 
$\widetilde B_0^{-1}$. Due to \eqref{cocycle} the family $h:=\{h_{ijk}\}$, $h_{ijk}:=h_{ij}\cdot h_{jk}\cdot h_{ki}$ on $U_i\cap U_j\cap U_k\ne\emptyset$, is a (continuous) $2$-cocycle on $\mathcal U$ with values in $Z$. Since by our assumption $h_{ijk}$ admits a continuous extension to the compact set $\overline{U_i\cap U_j\cap U_k}$, the image of each $h_{ijk}$ is a finite subset of $Z$. Thus, because the family $\{h_{ijk}\}$ is finite, images of all functions of this family generate a finitely-generated subgroup $Z'$ of $Z$. Hence, $h$ determines an element of the \v{C}ech cohomology group $H^2(M(H^\infty),Z')$. 
\begin{Lm}\label{vanish}
$H^2(M(H^\infty),Z')=0$.
\end{Lm}
\begin{proof}
Since $Z'$ is a direct product of finitely many copies of groups $\Z$ and $Z_{p^{m}}$, where $p\in\N$ is prime, it suffices to show that
$H^2(M(H^\infty),\Z)=0$ and $H^2(M(H^\infty),\Z_{p^m})=0$. The former group is trivial by \cite[Cor.~3.9]{S1}. Next, let $c=\{c_{ijk}\}\in Z^2({\mathcal V},\Z_{p^m})$ be a continuous cocycle defined on a finite open cover $\mathcal V=(V_i)_{i\in I}$ of $M(H^\infty)$. According to \cite{S1} the covering dimension of $M(H^\infty)$ is $2$, hence, passing to a refinement of $\mathcal V$, if necessary, without loss of generality we may assume that the order of the cover $\mathcal V$ is $3$. Since each $c_{ijk}:V_i\cap V_j\cap V_k\to\Z_{p^m}$ is a continuous locally constant function, there exists a continuous locally constant function $\widetilde c_{ijk}:V_i\cap V_j\cap V_k\to\Z$ such that $q(\widetilde c_{ijk})=c_{ijk}$, where $q:\Z\to Z_{p^m}$ is the quotient homomorphism, and $\widetilde c_{iik}=0$, $\widetilde c_{ijk}=-\widetilde c_{jik}=\widetilde c_{jki}$ for all $i,j,k$.
As the order of $\mathcal V$ is $3$, $\widetilde c=\{\widetilde c_{ijk}\}$ is a continuous integer-valued $2$-cocycle on $\mathcal V$. Due to \cite{S1}, $\widetilde c$ determines zero element of $H^2(M(H^\infty),\Z)\, (=0)$. Thus the restriction of $\widetilde c$ to a finite open refinement of $\mathcal V$ is a coboundary. To avoid abuse of notation we will assume that $\widetilde c$ is yet a coboundary on $\mathcal V$. Thus there exist a continuous cochain $\{\widetilde c_{ij}:V_i\cap V_j\to\Z\}_{i,j\in I}$ such that $\widetilde c_{ij}-\widetilde c_{jk}+\widetilde c_{ki}=\widetilde c_{ijk}$ on
$V_i\cap V_j\cap V_k\ne\emptyset$. This implies that $c_{ij}-c_{jk}+c_{ki}=c_{ijk}$ on
$V_i\cap V_j\cap V_k\ne\emptyset$, where $c_{st}:=q(\widetilde c_{st})$. Thus $c$ is a coboundary on $\mathcal V$.
\end{proof}

According to the lemma the restriction of the cocycle $h$ to a suitable refinement of $\mathcal U$ is a coboundary. As before, without loss of generality, we will assume that $h$ is yet a coboundary on $\mathcal U$. Hence, there exist a continuous cochain $\{\widetilde h_{ij}:U_i\cap U_j\to\ Z'\}_{i,j\in I}$ such that
\[
\widetilde h_{ij}\cdot \widetilde h_{jk}\cdot \widetilde h_{ki}=h_{ijk}\quad\text{on}\quad U_i\cap U_j\cap U_k\ne\emptyset .
\]
We set
\[
\widetilde g_{ij}:=h_{ij}\cdot \widetilde h_{ij}^{-1}\quad\text{on}\quad U_i\cap U_j\ne\emptyset .
\]
Then, as $Z'$ is a central subgroup of $\widetilde B_0^{-1}$, 
\[
\widetilde g_{ij}\cdot\widetilde g_{jk}\cdot\widetilde g_{ki}\equiv 1_{\widetilde B_0^{-1}}\quad\text{on}\quad U_i\cap U_j\cap U_k\ne\emptyset .
\]
In particular, $\widetilde g=\{\widetilde g_{ij}\}$, $\widetilde g_{ij}\in\mathcal O(U_i\cap U_j;\widetilde B_0^{-1})$, is a $1$-cocycle determining a holomorphic principal bundle on $M(H^\infty)$ with fibre $\widetilde B_0^{-1}$. By Theorem \ref{principal} this bundle is holomorphically trivial. Hence, there
exist $\widetilde g_i\in\mathcal O(U_i;\widetilde B_0^{-1})$, $i\in I$, such that
\[
\widetilde g_i^{-1}\cdot\widetilde g_j=\widetilde g_{ij}\quad\text{on}\quad U_i\cap U_j .
\]
Since $\psi(\widetilde g_{ij})=g_{ij}$ for $g_i:=\psi(\widetilde g_i)\in\mathcal O(U_i; B_0^{-1})$, $i\in I$, we obtain
\[
g_i^{-1}\cdot g_j=g_{ij}\quad\text{on}\quad U_i\cap U_j .
\]
This shows that the bundle $E$ is holomorphically trivial.

The proof of Theorem \ref{principal1} is complete.
\end{proof}
\sect{Proofs of Theorem \ref{bundle} and Corollary \ref{cor1.2}}
We will prove a more general result. Let $q: B^{-1}\to B^{-1}/B_{0}^{-1}=: C(B^{-1})$ be the (continuous) quotient homomorphism onto the discrete group of connected components of $B^{-1}$. Let $p:E\to M(H^\infty)$ be a holomorphic principal bundle with fibre $B^{-1}$ defined on a finite open cover $\mathcal U=(U_i)_{i\in I}$ of $M(H^\infty)$ by a cocycle $g=\{g_{ij}\in \mathcal O (U_i\cap U_j; B^{-1})\}$. By $E_{C(B^{-1})}\to M(H^\infty)$ we denote the principal bundle with fibre $C(B^{-1})$ defined on $\mathcal U$ by the locally constant cocycle $q(g)=\{q(g_{ij})\in C(U_i\cap U_j; C(B^{-1}))\}$.
\begin{Th}\label{bundle1}
$E$ is holomorphically trivial if and only if the associated bundle $E_{C(B^{-1})}$ is trivial (in the category of principal bundles with discrete fibres).
\end{Th}
\begin{proof}
If $E$ is holomorphically trivial, then there exist $g_i\in\mathcal O(U_i;B^{-1})$, $i\in I$, such that
\[
g_i^{-1}\cdot g_j=g_{ij}\quad\text{on}\quad U_i\cap U_j .
\]
This implies that 
\[
q(g_i)^{-1}\cdot q(g_j)=q(g_{ij})\quad\text{on}\quad U_i\cap U_j ,
\]
i.e., $E_{C(B^{-1})}$ is isomorphic to $M(H^\infty)\times C(B^{-1})$ in the category of principal bundles with discrete fibres.

Conversely, suppose $E_{C(B^{-1})}$ is isomorphic to $M(H^\infty)\times C(B^{-1})$ in the category of principal bundles with discrete fibres. Then there exist
$h_i\in C(U_i; C(B^{-1}))$, $i\in I$, such that
\begin{equation}\label{eq7.1}
h_i^{-1}\cdot h_j=q(g_{ij})\quad\text{on}\quad U_i\cap U_j.
\end{equation}
Taking a refinement of $\mathcal U$, if necessary, without loss of generality we may assume that each $h_i$ admits a continuous extension to the compact set $\bar U_i$, $i\in I$. Since $h_i:\bar U_i\to C(B^{-1})$ is continuous and $C(B^{-1})$ is discrete, the image of $h_i$ is finite. In particular, there exists a continuous locally constant function $\widetilde h_i :U_i\to B^{-1}$ such that $q\circ \widetilde h_i=h_i$. By definition each $\widetilde h_i\in\mathcal O(U_i;B^{-1})$. Let us define cocycle $\widetilde g=\{\widetilde g_{ij}\in  \mathcal O (U_i\cap U_j; B^{-1})\}$ by the formulas
\begin{equation}\label{eq7.2}
\widetilde g_{ij}:=\widetilde h_i\cdot g_{ij}\cdot\widetilde h_j^{-1}\quad\text{on}\quad U_i\cap U_j.
\end{equation}
Then $\widetilde g$ determines a holomorphic principal bundle $\widetilde E$ on $M(H^\infty)$ with fibre $B^{-1}$ isomorphic to the bundle $E$. Also, from \eqref{eq7.1} we obtain for all $i,j\in I$
\[
q(\widetilde g_{ij})=1_{C(B^{-1})}.
\]
Thus each $\widetilde g_{ij}$ maps $U_i\cap U_j$ into $B_0^{-1}$. In particular, $\widetilde g$ determines also a subbundle of $\widetilde E$ with fibre $B_0^{-1}$. According to Theorem \ref{principal1} this subbundle is holomorphically trivial. Thus there exist $\widetilde g_i\in\mathcal O(U_i;B_0^{-1})$, $i\in I$, such that
\[
\widetilde g_i^{-1}\cdot \widetilde g_j=\widetilde g_{ij}\quad\text{on}\quad U_i\cap U_j .
\]
From here and \eqref{eq7.2} we obtain for all $i,j\in I$
\[
(\widetilde g_i\cdot h_i)^{-1}\cdot (\widetilde g_j\cdot h_j)=g_{ij}\quad\text{on}\quad U_i\cap U_j .
\]
This shows that $E$ is holomorphically trivial.
\end{proof}
\begin{proof}[Proof of Theorem \ref{bundle}]
Since a holomorphic Banach vector bundle on $M(H^\infty)$ with fibre $X$ is associated with a holomorphic principal bundle on $M(H^\infty)$ with fibre $GL(X)$, the required result follows from Theorem \ref{bundle1} with $B:=L(X)$.
\end{proof}
\begin{proof}[Proof of Corollary \ref{cor1.2}]
(1) The result follows directly from Theorem \ref{principal1}.

(2) Let $E$ is defined on a finite open cover $\mathcal U=(U_i)_{i\in I}$ of $M(H^\infty)$ by a cocycle $g=\{g_{ij}\in \mathcal O (U_i\cap U_j; GL(X))\}$. Since $E$ is topologically trivial, there exist $g_i\in C(U_i;GL(X))$, $i\in I$, such that
\[
g_i^{-1}\cdot g_j=g_{ij}\quad\text{on}\quad U_i\cap U_j .
\]
This implies that 
\[
q(g_i)^{-1}\cdot q(g_j)=q(g_{ij})\quad\text{on}\quad U_i\cap U_j ,
\]
i.e., $E_{C(GL(X))}$ is isomorphic to $M(H^\infty)\times C(GL(X))$ in the category of principal bundles with discrete fibres. Thus the required result follows from Theorem \ref{bundle}.
\end{proof}

\sect{Proof of Theorem \ref{complem}}
\begin{proof}
Since $F\in H_{\rm comp}^\infty(L(X_1,X_2))$, by \cite[Prop. 1.3]{Br2} it is extended to a holomorphic function $M(H^\infty)\to L(X_1,X_2)$ (denoted by the same symbol). We set $M:=\max_{z\in M(H^\infty)}\|F(z)\|$.

\begin{Lm}\label{local1}
There exist a finite open cover $(U_i)_{1\le i\le k}$ of $M(H^\infty)$ and operators $G_i\in\mathcal O(U_i; L(X_2, X_1))$ such that
$G_i(z)F(z)=I_{X_1}$ for all $z\in U_i$.
\end{Lm}
\begin{proof}
Let
$C:=\sup_{z\in\Di}\|G_z\|$. Since $F$ is uniformly continuous on $M(H^\infty)$, for every $\varepsilon>0$ there exists a finite open cover $(U_i)_{1\le i\le k}$ of $M(H^\infty)$ such that 
\[
\max_{1\le i\le k}\sup_{x,y\in U_i}\|F(x)-F(y)\|<\varepsilon .
\]
Fix $x_i\in U_i\cap\Di$, $1\le i\le k$, and
define $A_i(z):=G_{x_i} F(z):X_1\to X_1$, $z\in U_i$. We have $A_i(x_i)=I_{X_1}$ and $\|A_i(z)-A_i(x_i)\|< C\varepsilon$. Choose $\varepsilon$ so small that $C\varepsilon<\frac 12$. Then $A_i^{-1}:=\sum_{j=0}^\infty (I_{X_1}-A_i)^j\in \mathcal O(U_i, L(X_1))$ is well defined, and $\|A_i^{-1}\|<2$. We set
\[
G_i(z):=A_i(z)^{-1}G_{x_i}.
\]
Clearly, $G_i(z)F(z)=A_i(z)^{-1}G_{x_i} F(z)=I_{X_1}$ and $\|G_i(z)\|< 2C$ for all $z\in U_i$.
\end{proof}
Choosing $\varepsilon<\frac{1}{8C(1+2MC)}$ sufficiently small we may obtain for all $z_1, z_2\in U_i$, $1\le i\le k$,
\begin{equation}\label{small1}
\begin{array}{l}
\|G_i(z_1)-G_i(z_2)\|\le C\|A_i(z_1)^{-1}-A_i(z_2)^{-1}\|\\
\\
\displaystyle
\le C\|A_i(z_1)-A_i(z_2)\|\cdot\left(1+\sum_{j=2}^\infty j\cdot (C\varepsilon)^{j-1}\right)<\frac{1}{4M(1+2MC)}.
\end{array}
\end{equation}

Let $Y_{iz}:={\rm Ker}(G_i(z))$, $P_{i}(z):=I_{X_2}-F(z)G_i(z)$, $z\in U_i$. Then $P_i(z)$ is a projector onto $Y_{iz}$.
For $z,w\in U_i$ consider the diagram
\[
Y_{iz}\stackrel{P_{i}(w)}{\longrightarrow}Y_{i w}\stackrel{P_i (z)}{\longrightarrow} Y_{iz}.
\]
If $H_{i}(z,w):=P_{i}(z)P_i(w)$, then $H_i(z,z)=I_{Y_{iz}}$ and by \eqref{small1}
$$
\begin{array}{l}
\|H_i(z,w)-H_i(z,z)\|\le\|P_i(z)\|\cdot\|F(w)G_i(w)-F(z)G_i(z)\|\\
\\
\displaystyle
\le  (1+2MC)\cdot \left(2C\varepsilon+\frac{M}{4M(1+2MC)}\right)<\frac{1}{2}.
\end{array}
$$
Thus $H_i(z,w)^{-1}$ exists for all $1\le i\le k$ and $z, w\in U_i$.
This implies that $P_i(z)$ is surjective and $P_i(w)$ is injective for all $z,w\in U_i$, i.e., all $P_i(z)$ are isomorphisms. 
Choosing as one of $x_i$ point $0$, we, hence, obtain (since
$M(H^\infty)$ is connected) that all $Y_{iz}$ are isomorphic to $Y:={\rm Ker}(G_0)$.
By $S_i: Y\to Y_{i x_i}$, $1\le i\le k$, we denote the corresponding isomorphisms.

Next, we define $B_i\in\mathcal O(U_i; L(X_1\oplus Y,X_2))$ as
\[
B_i(z)(v_1,v_2):=F(z)(v_1)+P_i(z)S_i(v_2),\quad z\in  U_{i}.
\]
Clearly, all $B_i(z)$ are isomorphisms.
We set for all $z\in U_i\cap U_j\ne\emptyset$
\[
B_{ij}(z):=B_i(z)^{-1}B_j(z): X_1\oplus Y\to X_1\oplus Y.
\]
Since $B_{ij}(z)(v)=v$ for all $v\in X_1$, 
\begin{equation}\label{matrix}
B_{ij}=\left(
\begin{array}{cc}
I_{X_1}&E_{ij}\\
0&D_{ij}
\end{array}
\right)
\end{equation}
for some $D_{ij}\in\mathcal O(U_i\cap U_j; GL(Y))$, $E_{ij}\in\mathcal O(U_i\cap U_j; L(Y,X_1))$. 

The holomorphic $1$-cocycle $\{D_{ij}\}$ determines a holomorphic Banach vector bundle on $M(H^\infty)$ with fibre $Y$. Since $Y$ satisfies conditions of Theorem \ref{bundle}, this bundle is holomorphically trivial. In particular, there exist $\widetilde D_i\in\mathcal O(U_i; GL(Y))$ such that
$\widetilde D_{i}\widetilde D_j^{-1}=D_{ij}$ on $U_i\cap U_j\ne\emptyset$. We set  $D_i:={\rm diag}\bigl(I_{X_1}, \widetilde D_i\bigr)$ on $U_i$ and define
\[
C_{ij}(z):=D_i(z)^{-1}B_{ij}(z)D_j(z),\quad z\in U_i\cap U_j\ne\emptyset .
\]
Then
$$
C_{ij}=\left(
\begin{array}{cc}
I_{X_1}&F_{ij}\\
0&I_Y
\end{array}
\right)
$$
for some $F_{ij}\in\mathcal O(U_i\cap U_j; L(Y,X_1))$. 

Clearly, $\{F_{ij}\}$ is an additive holomorphic $1$ cocycle on cover $(U_i)$ with values in $L(Y,X_1)$. According to \cite[Th.~1.4]{Br2}, $\{F_{ij}\}$ represents zero element in the corresponding \v{C}ech cohomology group.
This implies that there exist $\widetilde F_i\in \mathcal O(U_i; L(Y, X_1))$ such that $\widetilde F_i-\widetilde F_j=F_{ij}$ on $U_i\cap U_j\ne\emptyset$.

We set
$$
F_{i}=\left(
\begin{array}{cc}
I_{X_1}&\widetilde F_{i}\\
0&I_Y
\end{array}
\right)
$$
and define
\[
H(z):=B_i(z)D_i(z)F_i(z),\quad z\in U_i.
\]
Then $H\in H_{\rm comp}^\infty(L(H_1\oplus Y, H_2))$ has inverse $G\in H_{\rm comp}^\infty(L(H_2, H_1\oplus Y))$ and $H(z)|_{X_1}=F(z)$ for all $z$.
\end{proof}

\sect{Proof of Theorem \ref{teo1.5}}
\begin{proof}
We retain notation of the proof of Theorem \ref{complem}. According to \eqref{matrix} cocycle $\{B_{ij}\}$ determines the trivial holomorphic Banach vector bundle $E$ on $M(H^\infty)$ with fibre $X_1\oplus Y$ which has the trivial holomorphic Banach vector subbundle $E_1$ with fibre $X_1$ (determined by trivial cocycle $\{I_{X_1}\}$) and the holomorphic quotient Banach vector bundle $E_2$ with fibre $Y$ (determined by cocycle $\{D_{ij}\}$). Thus we have an exact sequence of bundles
\begin{equation}\label{eq9.1}
0\longrightarrow E_1\longrightarrow E\longrightarrow E_2\longrightarrow 0
\end{equation}
which induces an exact sequence of Banach holomorphic bundles
\[
0\longrightarrow Hom(E_2, E_1)\longrightarrow Hom(E_2, E)\longrightarrow Hom(E_2, E_2)\longrightarrow 0.
\]
In turn, the latter produces the long cohomology sequences of sheaves of germs of holomorphic sections of these bundles:
$$
\begin{array}{lr}
\displaystyle
0\longrightarrow H^0(M(H^\infty), Hom(E_2, E_1))\longrightarrow H^0(M(H^\infty),Hom(E_2, E))\\
\\
\displaystyle
\longrightarrow H^0(M(H^\infty), Hom(E_2, E_2))\stackrel{\delta}{\longrightarrow} H^1(M(H^\infty), Hom(E_2, E_1))\longrightarrow\cdots .
\end{array}
$$
The identity map $E_2\to E_2$ determines an element $I\in  H^0(M(H^\infty), Hom(E_2, E_2))$. Then sequence \eqref{eq9.1} splits if and only if
$\delta(I)=0\in H^1(M(H^\infty), Hom(E_2, E_1))$, for the basic results related to extensions of bundles, see, e.g., \cite{A}. In our notation
$\delta(I)$ is represented by the holomorphic $1$-cocycle $\{E_{ij}\}$, see \eqref{matrix}. Since $M(H^\infty)$ is paracompact, $\delta(I)=0$ in the corresponding cohomology group of sheaves of germs of {\em continuous} sections of $Hom(E_2,E_1)$. Therefore $E$ is isomorphic in the category of continuous bundles on $M(H^\infty)$ to the bundle $E_1\oplus E_2$ (the Whitney sum). Further, since $E$ is holomorphically trivial, $E_1\oplus E_2$ is the topologically trivial bundle. Thus according to Corollary \ref{cor1.2}\,(2), the holomorphic Banach vector bundle $E_1\oplus E_2$ is holomorphically trivial. By definition we have 
\[
Hom(E_1\oplus E_2,E_1)=Hom(E_1,E_1)\oplus Hom(E_2,E_1).
\]
Thus there exist homomorphisms 
\[
i: H^1(M(H^\infty), Hom(E_2, E_1))\to H^1(M(H^\infty), Hom(E_1\oplus E_2, E_1))
\] 
and
\[
j:H^1(M(H^\infty), Hom(E_1\oplus E_2, E_1))\to H^1(M(H^\infty), Hom(E_2, E_1))
\]
such that $j\circ i={\rm id}$. Consider $i(\delta(I))\in H^1(M(H^\infty), Hom(E_1\oplus E_2, E_1))$. Since $E_1\oplus E_2$ and $E_1$ are holomorphically trivial, the holomorphic Banach vector bundle $Hom(E_1\oplus E_2, E_1))$ is holomorphically trivial as well, i.e., it is holomorphically isomorphic to the bundle $M(H^\infty)\times L(X_1\oplus Y, X_1)$. Then $i(\delta(I))$ is naturally identified with an element of the cohomology group $H^1(M(H^\infty);L(X_1\oplus Y, X_1))$ of the sheaf of germs of holomorphic functions on $M(H^\infty)$ with values in the Banach space $L(X_1\oplus Y, X_1)$. According to \cite[Th. 1.4]{Br2}, $H^1(M(H^\infty);L(X_1\oplus Y, X_1))=0$. Hence,
$\delta(I)=j(i(\delta(I)))=0\in H^1(M(H^\infty), Hom(E_1\oplus E_2, E_1))$.

Thus we have proved that sequence \eqref{eq9.1} splits. This is equivalent to the fact that there exist $\widetilde E_i\in\mathcal O(U_i, L(Y,X_1))$, $i\in I$, such that
\[
\left(
\begin{array}{cc}
I_{X_1}&0\\
0&D_{ij}
\end{array}
\right)=
\left(
\begin{array}{cc}
I_{X_1}&\widetilde E_{i}\\
0&I_{Y}
\end{array}
\right)^{-1}\cdot
\left(
\begin{array}{cc}
I_{X_1}&E_{ij}\\
0&D_{ij}
\end{array}
\right)\cdot
\left(
\begin{array}{cc}
I_{X_1}&\widetilde E_{j}\\
0&I_{Y}
\end{array}
\right)
\quad\text{on}\quad U_i\cap U_j.
\]
We set
\[
E_i:=\left(
\begin{array}{cc}
I_{X_1}&\widetilde E_{i}\\
0&I_{Y}
\end{array}
\right)\in\mathcal O(U_i;GL(X_1\oplus Y)).
\]
Then 
\[
(B_i\cdot E_i)^{-1} (B_j\cdot E_j)={\rm diag}\, (I_{X_1}, D_{ij})\quad\text{on}\quad U_i\cap U_j,
\] 
and $B_i\cdot E_i\in \mathcal O(U_i; L(X_1\oplus Y, X_2))$ are invertible and such that $(B_i(z)\cdot E_i(z))|_{X_1}=F(z)$, $z\in U_i$, $i\in I$.
Let $P_{X_1}:={\rm diag}\,(I_{X_1},0): X_1\oplus Y\to X_1$ be the natural projection. We set
\[
G:=P_{X_1}\cdot (B_i\cdot E_i)^{-1}\quad\text{on}\quad U_i.
\]
Then on $U_i\cap U_j\ne\emptyset$ we have
$$
\begin{array}{l}
\displaystyle
P_{X_1}\cdot ((B_i\cdot E_i)^{-1})-P_{X_1}\cdot((B_j\cdot E_j)^{-1})\\
\\
\displaystyle
=P_{X_1}\cdot\bigl(\bigl((B_i\cdot E_i)^{-1} (B_j\cdot E_j)- I_{X_1\oplus Y}\bigr)\cdot (B_j\cdot E_j)^{-1}\bigr)
\\
\\
\displaystyle
=
{\rm diag}\,(I_{X_1},0)\cdot{\rm diag}\, (0, D_{ij}-I_Y)\cdot (B_j\cdot E_j)^{-1}=0.
\end{array}
$$
Thus $G\in\mathcal O(M(H^\infty); L(X_2, X_1))$. Moreover, by our construction $G(z)F(z)=I_{X_1}$ for all $z\in M(H^\infty)$.

This completes the proof of the theorem.
\end{proof}

\end{document}